\documentclass[11pt,reqno]{amsart}
\usepackage{graphicx}
\usepackage{float}
\usepackage{amsfonts}
\usepackage[caption = false]{subfig}
\usepackage{amsmath}
\usepackage{graphics}
\usepackage{float}
\usepackage{enumerate,comment}
\usepackage{mathtools}
\usepackage{amsthm}
\usepackage[english]{babel}
\usepackage{pgf}
\usepackage{tikz} 
\usetikzlibrary{shapes,arrows,automata}
\usetikzlibrary{automata}
\usepackage{amsfonts,amssymb}
\usepackage{tikz}
 \usetikzlibrary{shapes.geometric, arrows,positioning}

  \tikzstyle{startstop} = [rectangle, text width=3cm, minimum   height=1cm,text centered, draw=black]
   \tikzstyle{process} = [rectangle, text width=3cm, minimum height=1cm,    text centered, draw=black]
   \tikzstyle{decision} = [rectangle, text width=3cm, minimum   height=1cm,text centered, draw=black]
   \tikzstyle{arrow} = [thick,->,>=stealth]

\usepackage{mathrsfs}
\usepackage[utf8x]{inputenc}\usepackage[english]{babel}
\usepackage{soul}
\usepackage{cite}
\usepackage{relsize}
\usepackage{tikz}
\usetikzlibrary{shapes, arrows}
\numberwithin{equation}{section}
\newtheorem{theorem}{Theorem}[section]
\newtheorem{corollary}[theorem]{Corollary}
\newtheorem{lemma}[theorem]{Lemma}

\theoremstyle{definition}
\newtheorem{definition}[theorem]{Definition}
\theoremstyle{remark}

\numberwithin{equation}{section}
\DeclareMathOperator{\RE}{Re}

\begin{document}
	\title[Higher order differential subordinations for certain starlike functions]{Higher order differential subordinations for certain starlike functions} 
 
\author[N. Verma]{Neha Verma}
	\address{Department of Applied Mathematics, Delhi Technological University, Delhi--110042, India}
	\email{nehaverma1480@gmail.com}
 
 \author[S. S. Kumar]{S. Sivaprasad Kumar}
	\address{Department of Applied Mathematics, Delhi Technological University, Delhi--110042, India}
	\email{spkumar@dce.ac.in}

	

 \maketitle
\begin{abstract}
In this paper, we employ a novel second and third-order differential subordination technique to establish the sufficient conditions for functions to belong to the classes $\mathcal{S}^*_s$ and $\mathcal{S}^*_{\rho}$, where $\mathcal{S}^*_s$ is the set of all normalized analytic functions $f$ satisfying $ zf'(z)/f(z)\prec 1+\sin z$ and $\mathcal{S}^*_{\rho}$ is the set of all normalized analytic functions $f$ satisfying $ zf'(z)/f(z)\prec 1+\sinh^{-1} z$.



\end{abstract}
\maketitle
	
\section{Introduction}

\noindent Consider the class $\mathcal{H}(\mathbb{D})$, comprising analytic functions defined on the unit disk $\mathbb{D}:=\{z\in \mathbb{C}:|z|<1\}$. For any positive integer $n$ and $a\in \mathbb{C}$, we define
\begin{equation*}
\mathcal{H}[a,n]:=\{f\in \mathcal{H}:f(z)=a+a_nz^n+a_{n+1}z^{n+1}+a_{n+2}z^{n+2}+\ldots\}.
\end{equation*}
We assume that $\mathcal{A}:=\mathcal{A}_1$ represents the class containing normalized analytic functions defined on $\mathbb{D}$ in the following format:
\begin{equation*}
\mathcal{A}_n=\{f\in\mathcal{H}: f(z)=z+a_{n+1}z^{n+1}+a_{n+2}z^{n+2}+\ldots\}.
\end{equation*}
Additionally, we define $\mathcal{S}$ as a subset of $\mathcal{A}$ consisting of univalent functions, and $\mathcal{S}^*$ as a subset of $\mathcal{S}$ containing all starlike functions. A starlike function's analytic characterization is given by $\RE (zf'(z)/f(z))>0$ for $z\in \mathbb{D}$. Now, considering two functions $f$ and $g$, both belonging to the class $\mathcal{A}$, we define the subordination relation as $f\prec g$ if there exists a Schwarz function $w$ that fulfills the conditions: $w(0)=0$ and $|w(z)|\leq |z|$ such that $f(z)=g(w(z))$. Furthermore, if $g$ is univalent, the subordination $f\prec g$ holds if and only if $f(0)=g(0)$ and $f(\mathbb{D})\subseteq g(\mathbb{D})$.

Ma and Minda \cite{ma-minda} introduced the class $\mathcal{S}^*(\varphi)$, where $\varphi$ is an analytic univalent function with properties: $\RE\varphi(z)>0$, $\varphi(\mathbb{D})$ is symmetric about the real axis and starlike with respect to $\varphi(0)=1$ and $\varphi'(0)>0$. This function $\varphi$ is referred to as the Ma-Minda function. The class $\mathcal{S}^*(\varphi)$ is defined as:

\begin{equation}
\mathcal{S}^*(\varphi)=\bigg\{f\in \mathcal {A}:\dfrac{zf'(z)}{f(z)}\prec \varphi(z) \bigg\}.\label{mindaclass}
\end{equation}

Researchers came across a number of interesting subclasses of $\mathcal{S}^*$ for the proper selections of the function $\varphi (z)$ in \eqref{mindaclass}. We enlist below some of these subclasses in Table \ref{10 table}. In 2019, Cho et al. \cite{chosine} introduced the class  $\mathcal{S}^{*}_{s}$ by choosing $\varphi(z)=1+\sin z$ and studied its radius problems. Within this context, a function $f$ belongs to $\mathcal{S}^{*}_{s}$ if the ratio $zf'(z)/f(z)$ resides within $\Omega_s:=\{w \in \mathbb{C}:|\arcsin (w -1)|<1\}$, which is an eight-shaped domain situated in the right half-plane. Similarly, in 2022, Arora and Kumar \cite{kush} introduced the class $\mathcal{S}^{*}_{\rho}$ for $\varphi(z)=1+\sinh^{-1} z$. In this setting, a function $f$ is categorized as a member of $\mathcal{S}^{*}_{\rho}$ if the quotient $zf'(z)/f(z)$ resides within $\Omega_\rho:=\{w \in \mathbb{C}:|\sinh (w -1)|<1\}$, a petal-shaped region.

\begin{table}[!htbp]\label{10 table}
\centering
\caption{List of subclasses of $\mathcal{S}^{*}$}
\begin{tabular}{|c|c |c|c|} 
 \hline
 $\mathcal{S}^{*}(\varphi)$ & $\varphi(z)$ & Author(s) & Reference \\ [1ex] 
 \hline
  $\mathcal{S}^*[C,D]$ & $(1+Cz)/(1+Dz)$   & Janowski&\cite{1janowski}     \\
  \hline
  $\mathcal{S}^{*}_{SG}$ & $2/(1+e^{-z})$ & Goel and Kumar &\cite{goel} \\
  \hline
   $\mathcal{S}^{*}_{\varrho}$ & $1+ze^z$ & Kumar and Kamaljeet &\cite{kumar-ganganiaCardioid-2021}\\
   \hline
   $\mathcal{S}^{*}_{e}$ & $e^z$ &  Mendiratta et al.&\cite{mendi} \\
   \hline
    $\mathcal{S}^{*}_{q}$ & $z+\sqrt{1+z^2}$ &  Raina and Sok\'{o}\l &\cite{raina} \\
    \hline
    $\mathcal{S}^{*}_ L$&$\sqrt{1+z}$&  Sok\'{o}\l \ and Stankiewicz  &\cite{stan}   \\
   
 \hline
 
\end{tabular}
\label{table1}
\end{table}

The study of differential subordination is a more generalized version of differential inequalities on the real line to the complex plane. Miller and Mocanu \cite{miller} published their monograph titled ``Differential subordination and univalent functions". Using the concept of  admissibility conditions as mentioned in \cite{miller}, Saliu et al. \cite{saliu} dealt with first and second-order differential implication for $p(z)\prec (1+z)^{\alpha}$ where $\alpha\in(0,1]$. Similarly, Mushtaq et al. \cite{mushtaq} obtained the conditions on $\beta$ so that the first-order differential subordinations like $1+\beta z p'(z)/p^j(z)\prec \sqrt{1+cz}$ and $1+\beta z p'(z)/p^j(z)\prec 1+\sqrt{2}z+z^2/2$ implies $p(z)\prec e^z$ for $j=0,1$ and $2$. Furthermore, Naz et al. \cite{adibastarlikenessexponential} and Goel and Kumar \cite{goel} have successfully deduced numerous first-order subordination implications employing admissible functions for the exponential and sigmoid functions, respectively. Also, Madaan et al. \cite{madaan} established a series of first and second-order differential subordination implications for the lemniscate of Bernoulli, employing admissibility conditions. For a comprehensive exploration of the most recent advancements in this direction, one may see \cite{goelhigher,nehadiffexpo,ckms,priyankabelg}. Hence, drawing inspiration from the aforementioned considerations, this article distinctly directs its attention towards investigating second and third-order subordination implications for the class of starlike functions pertinent to the eight-shaped domain and second-order subordination implications  for the class of starlike functions pertinent to the petal-shaped domain in Sections 2 and 3, respectively. Subsequently, we introduce fundamental definitions and notations relevant to differential subordination.

\begin{definition}\cite{antoninoandmiller}
Let $\phi(r,s,t,u;z):\mathbb{C}^4\times \mathbb{D}\rightarrow\mathbb{C}$ and $h(z)$ be a univalent function in $\mathbb{D}$, if $p$ is an analytic function in $\mathbb{D}$ satisfying the third-order differential subordination
\begin{equation}\label{8 def3}
    \phi(p(z),zp'(z),z^2p''(z),z^3p'''(z);z)\prec h(z)
\end{equation}
then $p$ is called the solution of the differential subordination. The univalent function $q$ is said to be a dominant of the solutions of the differential subordination if $p\prec q$ for all $p$ satisfying \eqref{8 def3}. A dominant $\bar{q}$ that satisfies $\bar{q}\prec q$ for all dominants $q$ of \eqref{8 def3} is said to be the best dominant of \eqref{8 def3}, which is unique upto the rotations of $\mathbb{D}$.
\end{definition}
Moreover, suppose $Q$ be the set of analytic and univalent functions $q\in\overline{\mathbb{D}}\setminus \mathbb{E}(q)$, where
\begin{equation*}
   \mathbb{E}(q)=\{\zeta\in \partial \mathbb{D}:\lim_{z\rightarrow\zeta}q(z)=\infty\}
\end{equation*}
such that $q'(\zeta)\neq 0$ for $\zeta\in \partial \mathbb{D}\setminus \mathbb{E}(q)$. The subclass of $Q$ for which $q(0)=a$ is denoted by $Q(a)$.

\begin{lemma}\cite{antoninoandmiller}\label{lemmaformk}
Let $z_0\in \mathbb{D}$ and $r_0=|z_0|$. Let $f(z)=\sum_{k=n}^{\infty}a_kz^k$ be continuous on $\overline{\mathbb{D}}_{r_0}$ and analytic on $\mathbb{D}\cup\{z_0\}$ with $f(z)\neq 0$ and $n\geq 2$. If $|f(z_0)|=\max \{|f(z)|:z\in \overline{\mathbb{D}}_{r_0}\}$ and $|f'(z_0)|=\max\{|f'(z)|:z\in \overline{\mathbb{D}}_{r_0}\}$, then there exist real constants $m$, $k$ and $l$ such that
\begin{equation*}
      \frac{z_0f'(z_0)}{f(z_0)}=m,\quad 1+\frac{z_0f''(z_0)}{f'(z_0)}=k\quad \text{and}\quad 2+\RE\bigg(\frac{z_0f'''(z_0)}{f''(z_0)}\bigg)=l
\end{equation*}
where $l\geq k\geq m\geq n\geq2$.
\end{lemma}

\begin{lemma}\cite[Theorem 2.3b]{miller}\label{10 millertheorem}
  Let $\phi\in \Psi_n[\Omega,q]$ with $q(0)=a$. If $p\in \mathcal{H}[a,n]$ satisfies
  \begin{equation*}
      \phi(p(z),zp'(z),z^2p''(z);z)\in \Omega,
  \end{equation*}
  then $p\prec q$.
\end{lemma}
Note that if $\Omega\subset \mathbb{C}$ is a simply connected domain, then there exists a conformal mapping $h$ from $\mathbb{D}$ onto $\Omega=h(\mathbb{D})$. 
Moreover, if the function $\phi(p(z),zp'(z),z^2p''(z);z)$ is analytic in $\mathbb{D}$, then $\phi(p(z),zp'(z),z^2p''(z);z)\in \Omega$ can be expressed in terms of subordination as $$\phi(p(z),zp'(z),z^2p''(z);z)\prec h(z).$$
For the Ma-Minda functions to satisfy the criteria for third-order differential subordination, Kumar and Goel \cite{goelhigher} modified the results of Antonino and Miller \cite{antoninoandmiller} in 2020. Moreover, they derived the first, second and third-order differential subordination implications for the class $\mathcal{S}^*_{SG}$, using these modified results through admissibility conditions. Recently, Verma and Kumar \cite{nehadiffexpo} have obtained results pertaining to second and third-order differential subordination for the class $\mathcal{S}^{*}_e$. Thus, this study is motivated to establish parameter-specific conditions enabling the validity of second and third-order subordination implications for the class $\mathcal{S}^{*}_s$, as well as second-order subordination implications for the class $\mathcal{S}^{*}_\rho$. 

The modified definition and lemma of the third-order differential subordination implication by Kumar and Goel \cite{goelhigher}, accompanied by some lemmas are presented below, which are required to deduce our results in the subsequent sections.

\begin{definition}\cite{goelhigher}
Let $\Omega$ be a set in $\mathbb{C}$, $q\in Q$ and $k\geq m\geq n\geq 2$. The class of admissible operators $\Psi_n[\Omega,q]$ consists of those $\phi:\mathbb{C}^4\times \mathbb{D}\rightarrow \mathbb{C}$ that satisfy the admissibility conditions
\begin{equation*}
 \phi(r,s,t,u;z)\notin \Omega\quad \text{whenever}\quad z\in \mathbb{D},    
\end{equation*}
\begin{equation*}
 r=q(\zeta),\quad s=m\zeta q'(\zeta), \quad \RE\bigg(1+\frac{t}{s}\bigg)\geq m\bigg(1+\RE \frac{\zeta q''(\zeta)}{q'(\zeta)}\bigg)
\end{equation*} 
and
\begin{equation*}
\RE \frac{u}{s}\geq m^2\RE \frac{\zeta^2 q'''(\zeta)}{q'(\zeta}+3m(k-1)\RE \frac{\zeta q''(\zeta)}{q'(\zeta)} \quad \text{for}\quad \zeta\in \partial \mathbb{D}\setminus \mathbb{E}(q).
\end{equation*}
\end{definition}

\begin{lemma}\cite{goelhigher} \label{8 firsttheoremthirdorder}
Let $p\in \mathcal{H}[a,n]$ with $m\geq n\geq2$, and let $q\in Q(a)$ such that it satisfies
\begin{equation*}
       \bigg|\frac{zp'(z)}{q'(\zeta)}\bigg|\leq m\quad \text{for}\quad z\in \mathbb{D}\quad \text{and}\quad \zeta\in \partial \mathbb{D}\setminus \mathbb{E}(q).
\end{equation*}
If $\Omega$ is a set in $\mathbb{C}$, $\phi\in \Psi_n[\Omega,a]$ and
\begin{equation*}
\phi(p(z),zp'(z),z^2p''(z),z^3p'''(z);z)\subset \Omega,
\end{equation*}
then $p\prec q$.
\end{lemma}

\begin{lemma}\cite{goel}\label{prilemma42}
  Let $r_0\approx 0.546302$ be the positive root of the equation $r^2+2 \cot(1)r-1=0$. Then
  \begin{equation*}
      \bigg|\log \bigg(\frac{1+z}{1-z}\bigg)\bigg|\geq 1\quad \text{on}\quad |z|=R\quad \text{if and only if}\quad R\geq r_0.
  \end{equation*}
\end{lemma}

\begin{lemma}\cite[Lemma 4, Pg No. 192]{goelhigher}\label{prilemma4}
For any complex number $z$, we have
\begin{equation*}
    |\log (1+z)|\geq 1\quad \text{if and only if} \quad |z|\geq e-1.
\end{equation*}
\end{lemma}

\section{The Class $\mathcal{S}^{*}_s$}

\noindent We begin our study with the function $q(z):=1+\sin z$ and define the admissibility class $\Psi[\Omega,q]$, where $\Omega\subset \mathbb{C}$. We know that $q(z)$ is analytic and univalent on $\overline{\mathbb{D}}$, $q(0)=1$ and it maps $\mathbb{D}$ onto the domain $\Omega_{s}:=\{w \in \mathbb{C}:|\arcsin (w -1)|<1\}$. Since $\mathbb{E}(q)=\phi$, for $\zeta\in \partial \mathbb{D}\setminus \mathbb{E}(q)$ if and only if $\zeta=e^{i\theta}$ for $\theta\in[0,2\pi]$. Now, consider
\begin{equation}
    |q'(\zeta)|=\sqrt{\cosh^2(\sin\theta)-\sin^2(\cos\theta)}=:n_1(\theta)\label{10s 1}
\end{equation}
and $n_1(\theta)$ achieves its minimum value at $\theta=0$, denoted by
\begin{equation}\label{10s nu0}
   \nu_0:= n_1(0)=\sqrt{1-\sin^{2} 1}\approx 0.540302.
\end{equation}
It is evident that $\min |q'(\zeta)|>0$, which implies that $q\in Q(1)$, and consequently, the admissibility class $\Psi[\Omega,q]$ is well-defined. While considering $|\zeta|=1$, we observe that $q(\zeta)\in q(\partial \mathbb{D})=\partial \Omega_{s} = \{w  \in \mathbb{C}: |\arcsin(w -1)| = 1\}$. Consequently, $|\arcsin (q(\zeta)-1)| = 1$ and $\arcsin (q(\zeta)-1) = e^{i\theta}$ for $ \theta\in [0, 2\pi]$, which implies that $q(\zeta) = 1+\sin\zeta$. Moreover, $\zeta q'(\zeta)=e^{i\theta}\cos(e^{i\theta})$ and 
\begin{equation}
\frac{\zeta q''(\zeta)}{q'(\zeta)}=e^{i\theta}\tan(e^{i\theta}).\label{10s 2}
\end{equation}
By comparing the real parts on both sides of \eqref{10s 2}, we obtain
\begin{equation}
    \RE\bigg(\frac{\zeta q''(\zeta)}{q'(\zeta)}\bigg)=\frac{-\cos \theta \sin(2\cos\theta)+\sin \theta \sinh (2\sin \theta)}{\cos (2\cos \theta)+\cosh(2\sin \theta)}=:n_2(\theta).\label{10s 3} 
\end{equation}
The function $n_2(\theta)$ as defined in \eqref{10s 3}, attains its minimum at $\theta=0$, denoted by
\begin{equation}
    \nu_1:=n_2(0)=-\frac{\sin 2}{1+\cos 2}\approx -1.55741.\label{10s nu1}
\end{equation} 
Moreover, the class $\Psi[\Omega, 1+\sin z]$ is precisely defined as the class of all functions $\phi: \mathbb{C}^3 \times \mathbb{D} \rightarrow \mathbb{C}$ that satisfy the following conditions:
\begin{equation*}
    \phi(r,s,t;z)\notin \Omega \quad \text{for}\quad z\in \mathbb{D},\quad \theta\in[0,2\pi] \quad\text{and}\quad m\geq 1,
\end{equation*}
whenever
\begin{equation}
    r=q(\zeta)=1+\sin(e^{i\theta});\quad s=m\zeta q'(\zeta)=me^{i\theta}\cos(e^{i\theta});\quad \RE\bigg(1+\frac{t}{s}\bigg)\geq m(1+n_2(\theta)). \label{10s 4}
\end{equation}

Furthermore, for $q(z)=1+\sin z$, we have
\begin{equation*}
    \zeta^2\frac{q'''(\zeta)}{q'(\zeta)}=-e^{2i\theta}.
\end{equation*}
On comparing the real parts of both the sides,
\begin{equation*}
    \RE\bigg(\zeta^2\frac{q'''(\zeta)}{q'(\zeta)}\bigg)=-\cos 2\theta=:n_3(\theta). 
\end{equation*}
The minimum value of $n_3(\theta)$ is $-1$, which is attained at $\theta=0$. Thus, if $\phi:\mathbb{C}^4\times \mathbb{D}\rightarrow\mathbb{C}$ then $\phi\in \Psi[\Omega,1+\sin z]$, provided $\phi$ satisfies the following conditions:
\begin{equation*}
    \phi(r,s,t,u;z)\notin \Omega\quad \text{for}\quad z\in \mathbb{D},\quad \theta\in[0,2\pi] \quad\text{and}\quad k\geq m\geq2
\end{equation*}
whenever
\begin{equation*}
r=q(\zeta)=1+\sin{e^{i\theta}};\quad s=m\zeta q'(\zeta)=me^{i\theta}\cos(e^{i\theta});
\end{equation*}
\begin{equation*}
   \RE\bigg(1+\frac{t}{s}\bigg)\geq m(1+n_2(\theta)) \quad\text{and}\quad \RE \frac{u}{s}\geq m^{2}n_3(\theta)+3m(k-1)n_2(\theta).
\end{equation*}

We now find conditions on the parameters $\beta_1$ and $\beta_2$ for different choices of $h(z)$ namely $\sqrt{1+z}$, $(1+Cz)/(1+Dz)$, $2/(1+e^{-z})$, $z+\sqrt{1+z^2}$, $1+\sinh^{-1}z$, $1+ze^z$, $e^z$ and $1+\sin z$, so that the following second order differential subordination implication holds:
\begin{equation*}
    1+\beta_1 zp'(z)+\beta_2 z^2p''(z)\prec h(z) \implies p(z)\prec 1+\sin z
\end{equation*}
Note that the values of $\nu_0$ and $\nu_1$ are given in \eqref{10s nu0} and \eqref{10s nu1} respectively, which are widely used in the upcoming results. We begin with the following Theorem, for the case when $h(z) = \sqrt{1+z}.$

\begin{theorem}\label{10s secondorder1}
Let $\beta_1$, $\beta_2>0$ and $\nu_0({\beta_1}+ \beta_2 \nu_1)(\nu_0(\beta_1+\beta_2 \nu_1)-2)\geq 1$. Let $p$ be analytic function in $\mathbb{D}$ with $p(0)=1$ and
\begin{equation*}
      1+\beta_1 zp'(z)+\beta_2 z^2 p''(z)\prec \sqrt{1+z}.
\end{equation*}
Then $p(z)\prec 1+\sin z$.
\end{theorem}

\begin{proof}
Suppose $h(z)=\sqrt{1+z}$ for $z\in \mathbb{D}$. Then, $h(\mathbb{D})=\{w \in \mathbb{C}:|w ^2-1|<1\}=:\Omega$. Let $\phi:\mathbb{C}^3\times \mathbb{D}\rightarrow \mathbb{C}$ be defined as $\phi(r,s,t;z)=1+\beta_1 s+\beta_2 t$. It is clear that $\phi\in \Psi[\Omega,1+\sin z]$ provided $\phi(r,s,t,;z)\notin \Omega$ for $z\in \mathbb{D}$. Note that
\begin{align*}
        |(\phi(r,s,t;z))^2-1|&=|(1+\beta_1 s+\beta_2 t)^2-1|\\
        &\geq |\beta_1 s+\beta_2 t|(|\beta_1 s+\beta_2 t|-2)\\
        &\geq |\beta_1 s|\RE\bigg(1+\frac{\beta_2}{\beta_1}\frac{t}{s}\bigg)\bigg(|\beta_1 s|\RE\bigg(1+\frac{\beta_2}{\beta_1}\frac{t}{s}\bigg)-2\bigg)\\
         &\geq L(L-2),
\end{align*}
where $L= m\beta_1 n_1(\theta)\RE(1+\beta_2(mn_2(\theta)+m-1)/{\beta_1})$. Here $n_1(\theta)$ and $n_2(\theta)$ are given in \eqref{10s 1} and \eqref{10s 3}, respectively. Since $m\geq 1$, we have
\begin{align*}
       |(\phi(r,s,t;z))^2-1|&\geq \beta_1 n_1(\theta)\RE\bigg(1+\frac{\beta_2}{\beta_1}n_2(\theta)\bigg)\bigg(\beta_1 n_1(\theta)\RE\bigg(1+\frac{\beta_2}{\beta_1}n_2(\theta)\bigg)-2\bigg)\\ 
       &\geq \nu_0(\beta_1+\beta_2 \nu_1) (\nu_0(\beta_1+\beta_2 \nu_1)-2)\\
       &\geq 1.
\end{align*}
Therefore, $\phi(r,s,t;z)\notin \Omega$ and hence $\phi\in \Psi[\Omega,1+\sin z]$. Now, the result follows at once by an application of Lemma \ref{10 millertheorem}.
\end{proof}

By considering the function $p(z)=zf'(z)/f(z)$ in Theorem \ref{10s secondorder1}, we derive the following corollary:

\begin{corollary}\label{10s corollaryfirst}
Suppose $\beta_1$, $\beta_2>0$ and $f\in\mathcal{A}$. Let
\begin{align}
 S_f(z)&:=1+\beta_1(S_2-{S_1}^2+S_1)+\beta_2(S_3+2S_2+2{S_1}^3-2{S_1}^2-3S_1S_2),  \label{10s corollary21}
\end{align}
where 
\begin{equation}
    S_1=\frac{zf'(z)}{f(z)}, \quad S_2=\frac{z^2f''(z)}{f(z)}\quad \text{and}\quad S_3=\frac{z^3f'''(z)}{f(z)}. \label{10s corollary21coefficients}
\end{equation}
Then $f\in \mathcal{S}^{*}_{s}$ provided
    $S_{f}(z)\prec \sqrt{1+z}$ and $\nu_0({\beta_1}+ \beta_2 \nu_1)(\nu_0(\beta_1+\beta_2 \nu_1)-2)\geq 1$.
\end{corollary}

\begin{theorem}
Let $\beta_1$, $\beta_2>0$ and $-1< D<C\leq 1$ with $\nu_0(\beta_1+\beta_2 \nu_1)(1-D^2)\geq(C-D)(1+|D|)$. Let $p$ be analytic function in $\mathbb{D}$ with $p(0)=1$ and
\begin{equation*}
        1+\beta_1 zp'(z)+\beta_2 z^2p''(z)\prec \frac{1+Cz}{1+Dz}.
\end{equation*}
Then $p(z)\prec 1+\sin z$.
\end{theorem}

\begin{proof}
Let $h(z)=(1+Cz)/(1+Dz)$ for $z\in \mathbb{D}$, then we have $h(\mathbb{D})=\{w\in \mathbb{C}:|w-(1-CD)/(1-D^2)|<(C-D)/(1-D^2)\}=:\Omega$. Let $\phi:\mathbb{C}^3\times \mathbb{D}\rightarrow \mathbb{C}$ be defined as $\phi(r,s,t;z)=1+\beta_1 s+\beta_2 t$. It is clear that $\phi\in \Psi[\Omega,1+\sin z]$ only when $\phi(r,s,t;z)\notin \Omega$ for $z\in \mathbb{D}$. Consider
\begin{align*}
        \bigg|\phi(r,s,t;z)-\frac{1-CD}{1-D^2}\bigg|&=\bigg|1+\beta_1 s+\beta_2 t-\frac{1-CD}{1-D^2}\bigg|\\
        &\geq |\beta_1 s+\beta_2 t|-\frac{|D|(C-D)}{1-D^2}\\
        &\geq |\beta_1 s|\RE\bigg(1+\frac{\beta_2}{\beta_1}\frac{t}{s}\bigg)-\frac{|D|(C-D)}{1-D^2}\\
        &\geq m\beta_1 n_1(\theta)\RE\bigg(1+\frac{\beta_2}{\beta_1}(mn_2(\theta)+m-1)\bigg)-\frac{|D|(C-D)}{1-D^2}.
\end{align*}
Using the fact that $m\geq 1$ and proceeding as in the proof of Theorem \ref{10s secondorder1}, we have
\begin{align*}
   \bigg|\phi(r,s,t;z)-\frac{1-CD}{1-D^2}\bigg|&\geq  
   \nu_0(\beta_1+\beta_2 \nu_1)-\frac{|D|(C-D)}{1-D^2}\\
   &\geq \frac{C-D}{1-D^2}.
\end{align*}
Thus, $\phi(r,s,t;z)\notin \Omega$ and hence $\phi\in \Psi[\Omega,1+\sin z]$. Now, the result follows by an application of Lemma \ref{10 millertheorem}.
\end{proof}

\begin{theorem}\label{10s thmsigmoid}
Suppose $\beta_1$, $\beta_2>0$ and $\nu_0(\beta_1+\beta_2\nu_1)\geq r_0$, where $r_0\approx 0.546302$ is the positive root of the equation $r^2+2 \cot (1)r-1=0$. Let $p$ be analytic function in $\mathbb{D}$ with $p(0)=1$ and
\begin{equation*}
1+\beta_1 zp'(z)+\beta_2 z^2p''(z)\prec \frac{2}{1+e^{-z}}.        
\end{equation*}
Then $p(z)\prec 1+\sin z$.
\end{theorem}

\begin{proof}
Assuming $h(z)=2/(1+e^{-z})$ for $z\in \mathbb{D}$. Then, $h(\mathbb{D})=\{w \in \mathbb{C}:|\log (w /(2-w ))|<1\}=:\Omega$. Let $\phi:\mathbb{C}^3\times \mathbb{D}\rightarrow \mathbb{C}$ be defined as $\phi(r,s,t;z)=1+\beta_1 s+\beta_2 t$. We know that $\phi\in \Psi[\Omega, 1+\sin z]$ provided $\phi(r,s,t;z)\notin \Omega$. First we consider
\begin{align*}
    |\beta_1 s+\beta_2 t|&=\beta_1 |s|\bigg|1+\frac{\beta_2}{\beta_1}\frac{t}{s}\bigg|\\
    &\geq \beta_1 |s|\RE\bigg(1+\frac{\beta_2}{\beta_1}\frac{t}{s}\bigg)\\
    &\geq m \beta_1 n_1(\theta)\bigg(1+\frac{\beta_2}{\beta_1}(mn_2(\theta)+m-1)\bigg).
\end{align*}
As $m\geq 1$, we obtain
\begin{align}\label{10s 7}
    |\beta_1 s+\beta_2 t|&\geq n_1(\theta)(\beta_1+\beta_2 n_2(\theta))\nonumber\\
    &\geq \nu_0(\beta_1+\beta_2 \nu_1)\nonumber\\
    &\geq r_0.
\end{align}
Now, we consider
\begin{equation*}
    \bigg|\log \bigg(\frac{\phi(r,s,t;z)}{2-\phi(r,s,t;z)}\bigg)\bigg|=\bigg|\log\bigg(\frac{1+\beta_1 s+\beta_2 t}{1-(\beta_1 s+\beta_2 t)}\bigg)\bigg|.
\end{equation*}
Through Lemma \ref{prilemma42} and \eqref{10s 7}, we have
\begin{equation*}
    \bigg|\log \bigg(\frac{1+\beta_1 s+\beta_2 t}{1-(\beta_1 s+\beta_2 t)}\bigg)\bigg|\geq 1,
\end{equation*}
which implies that $\phi\in \Psi[\Omega,1+\sin z]$. Now, the result follows by an application of Lemma \ref{10 millertheorem}.
\end{proof}

\begin{theorem}
Suppose $\beta_1$, $\beta_2>0$ and $\nu_0(\beta_1+\beta_2\nu_1)\geq \sqrt{2}$. Let $p$ be analytic function in $\mathbb{D}$ with $p(0)=1$ and
\begin{equation*}
        1+\beta_1 zp'(z)+\beta_2 z^2p''(z)\prec z+\sqrt{1+z^2}.
\end{equation*}
Then $p(z)\prec 1+\sin z$.
\end{theorem}

\begin{proof}
Let $h(z)=z+\sqrt{1+z^2}$ for $z\in \mathbb{D}$. Then, $h(\mathbb{D})=\{w \in \mathbb{C}:|w ^2-1|<2|w |\}=:\Omega$. Let $\phi:\mathbb{C}^3\times \mathbb{D}\rightarrow \mathbb{C}$ be defined as $\phi(r,s,t;z)=1+\beta_1 s+\beta_2 t$. For $\phi\in \Psi[\Omega, 1+\sin z]$, we must have $\phi(r,s,t;z)\notin \Omega$. From the geometry of $z+\sqrt{1+z^2}$ (see Fig. \ref{crescent10}), we note that $\Omega$ is constructed by the circles
\begin{equation*}
        C_1:|z-1|=\sqrt{2}\quad\text{and}\quad C_2:|z+1|=\sqrt{2}.
\end{equation*}

\begin{figure}[htp]
   \centering
   \includegraphics[width=0.7\textwidth]{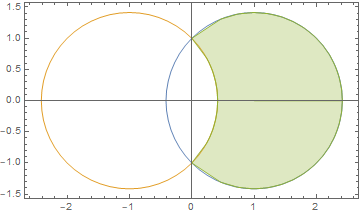}
    \caption{Graph of two circles, namely $C_1$ (blue boundary) and $C_2$ (orange boundary). While the shaded region represents $z+\sqrt{1+z^2}$.}
   \label{crescent10}
\end{figure}
It is obvious that $\Omega$ contains the disk enclosed by $C_1$ and excludes the portion of the disk enclosed by $C_1\cap C_2$. We have

\begin{equation*}
        |\phi(r,s,t;z)-1|=|\beta_1 s+\beta_2 t|.
\end{equation*}
Now, proceeding in the similar way as in the proof of Theorem \ref{10s thmsigmoid}, we have
\begin{align*}
       |\beta_1 s+\beta_2 t|&\geq n_1(\theta)(\beta_1+\beta_2 n_2(\theta))\\
       &\geq \nu_0(\beta_1+\beta_2\nu_1)\\
       &\geq \sqrt{2}.
\end{align*}
The fact that $\phi(r,s,t;z)$ lies outside the circle $C_1$ suffices us to deduce that $\phi(r,s,t;z)\notin \Omega$ and hence $\phi\in \Psi[\Omega,1+\sin z]$. Now, the result follows by an application of Lemma \ref{10 millertheorem}.
\end{proof}

\begin{theorem}
Let $\beta_1$, $\beta_2>0$ and $\nu_0(\beta_1+\beta_2\nu_1)\geq e$. Let $p$ be analytic function in $\mathbb{D}$ with $p(0)=1$ and
\begin{equation*}
        1+\beta_1 zp'(z)+\beta_2 z^2p''(z)\prec 1+ze^z.
\end{equation*}
Then $p(z)\prec 1+\sin z$.
\end{theorem}

\begin{proof}
Assuming $h(z)=1+ze^z$ for $z\in \mathbb{D}$. Then, $\Omega:=h(\mathbb{D})$. Let $\phi:\mathbb{C}^3\times \mathbb{D}\rightarrow \mathbb{C}$ be defined as $\phi(r,s,t;z)=1+\beta_1 s+\beta_2 t$. For $\phi\in \Psi[\Omega, 1+\sin z]$, we need to have $\phi(r,s,t;z)\notin \Omega$. Through \cite[Lemma 3.3]{kumar-ganganiaCardioid-2021},
we observe that the smallest disk containing $\Omega$ is $\{\delta \in \mathbb{C}: |\delta -1|<e\}$. Thus
\begin{equation*}
        |\phi(r,s,t;z)-1|=|\beta_1 s+\beta_2 t|.
\end{equation*}
Analogous to Theorem \ref{10s thmsigmoid}, we get
\begin{align*}
       |\beta_1 s+\beta_2 t|&\geq n_1(\theta)(\beta_1+\beta_2 n_2(\theta))\\
       &\geq \nu_0(\beta_1+\beta_2\nu_1)\\
       &\geq e.
\end{align*}
Clearly, $\phi(r,s,t;z)$ lies outside the disk $\{\delta \in \mathbb{C}: |\delta -1|<e\}$ which suffices us to conclude that $\phi(r,s,t;z)\notin \Omega$. Therefore, $\phi\in \Psi[\Omega,1+\sin z]$ and the result follows as an application of Lemma \ref{10 millertheorem}.
\end{proof}

\begin{theorem}
Suppose $\beta_1$, $\beta_2>0$ and $2\nu_0(\beta_1+\beta_2\nu_1)\geq \pi$. Let $p$ be analytic function in $\mathbb{D}$ with $p(0)=1$ and
\begin{equation*}
        1+\beta_1 zp'(z)+\beta_2 z^2p''(z)\prec 1+\sinh^{-1} z.
\end{equation*}
Then $p(z)\prec 1+\sin z$.
\end{theorem}

\begin{proof}
Let $h(z)=1+\sinh^{-1}z$ for $z\in \mathbb{D}$. Then, $h(\mathbb{D})=\{w \in\mathbb{C}:|\sinh (w -1)|<1\}=:\Omega_\rho$. Let $\phi:\mathbb{C}^3\times \mathbb{D}\rightarrow \mathbb{C}$ be defined as $\phi(r,s,t;z)=1+\beta_1 s+\beta_2 t$. For $\phi\in \Psi[\Omega_\rho, 1+\sin z]$, we must have $\phi(r,s,t;z)\notin \Omega_\rho$. Through \cite[Remark 2.7]{kush}, we note that the disk $\{\delta \in\mathbb{C}: |\delta -1|<\pi/2\}$ is the smallest disk containing $\Omega_\rho$. So,
\begin{equation*}
        |\phi(r,s,t;z)-1|=|\beta_1 s+\beta_2 t|.
\end{equation*}
Analogous to Theorem \ref{10s thmsigmoid}, we have
\begin{align*}
       |\beta_1 s+\beta_2 t|&\geq n_1(\theta)(\beta_1+\beta_2 n_2(\theta))\\
       &\geq \nu_0(\beta_1+\beta_2\nu_1)\\
       &\geq \frac{\pi}{2}.
\end{align*}
Clearly, $\phi(r,s,t;z)$ does not belong to the disk $\{\delta \in \mathbb{C}: |\delta -1|<\pi/2\}$ which suffices us to conclude that $\phi(r,s,t;z)\notin \Omega_\rho$. Therefore, $\phi\in \Psi[\Omega_\rho,1+\sin z]$ and thus the result follows as an application of Lemma \ref{10 millertheorem}.
\end{proof}


\begin{theorem}
Let $\beta_1$, $\beta_2>0$ and $\nu_0(\beta_1+\beta_2\nu_1)\geq e-1$. Let $p$ be analytic function in $\mathbb{D}$ with $p(0)=1$ and
\begin{equation*}
        1+\beta_1 zp'(z)+\beta_2 z^2p''(z)\prec e^z.
\end{equation*}
Then $p(z)\prec 1+\sin z$.
\end{theorem}

\begin{proof}
Suppose $h(z)=e^z$ for $z\in \mathbb{D}$. Then, $h(\mathbb{D})=\{w \in\mathbb{C}:|\log w |<1\}=:\Omega$. Let $\phi:\mathbb{C}^3\times \mathbb{D}\rightarrow \mathbb{C}$ be defined as $\phi(r,s,t;z)=1+\beta_1 s+\beta_2 t$. For $\phi\in \Psi[\Omega, 1+\sin z]$, we must have $\phi(r,s,t;z)\notin \Omega$. So,
\begin{equation*}
        |\phi(r,s,t;z)-1|=|\beta_1 s+\beta_2 t|.
\end{equation*}
Proceeding similar to the proof of Theorem \ref{10s thmsigmoid}, we have
\begin{align}\label{10s 82}
       |\beta_1 s+\beta_2 t|&\geq n_1(\theta)(\beta_1+\beta_2 n_2(\theta))\nonumber\\
       &\geq \nu_0(\beta_1+\beta_2\nu_1)\nonumber\\
       &\geq e-1.
\end{align}
Further, we have
\begin{equation*}
     |\log (\phi(r,s,t;z)|=|\log (1+\beta_1 s+\beta_2 t)|.
\end{equation*}
Through Lemma \ref{prilemma4} and \eqref{10s 82}, we have
\begin{equation*}
   |\log (1+\beta_1 s+\beta_2 t)|\geq 1, 
\end{equation*}
which implies that $\phi(r,s,t;z)\notin \Omega$. Therefore, $\phi\in \Psi[\Omega,1+\sin z]$ and thus the result follows as an application of Lemma \ref{10 millertheorem}.
\end{proof}

\begin{theorem}
Suppose $\beta_1$, $\beta_2>0$ and $\nu_0(\beta_1+\beta_2\nu_1)\geq \sinh 1$. Let $p$ be analytic function in $\mathbb{D}$ with $p(0)=1$ and
\begin{equation*}
        1+\beta_1 zp'(z)+\beta_2 z^2p''(z)\prec 1+\sin z.
\end{equation*}
Then $p(z)\prec 1+\sin z$.
\end{theorem}

\begin{proof}
Suppose $h(z)=1+\sin z$ for $z\in \mathbb{D}$. Then, $h(\mathbb{D})=\{w \in \mathbb{C}:|\arcsin (w -1)|<1\}=:\Omega_s$. Let $\phi:\mathbb{C}^3\times \mathbb{D}\rightarrow \mathbb{C}$ be defined as $\phi(r,s,t;z)=1+\beta_1 s+\beta_2 t$. For $\phi\in \Psi[\Omega_s, 1+\sin z]$, we must have $\phi(r,s,t;z)\notin \Omega_s$. Through \cite[Lemma 3.3]{chosine}, we note that $\{\delta \in \mathbb{C}: |\delta -1|<\sinh 1\}$ is the smallest disk containing $\Omega_s$. So,
\begin{equation*}
        |\phi(r,s,t;z)-1|=|\beta_1 s+\beta_2 t|.
\end{equation*}
Analogous to Theorem \ref{10 thmsigmoid}, we have
\begin{align*}
       |\beta_1 s+\beta_2 t|&\geq n_1(\theta)(\beta_1+\beta_2 n_2(\theta))\\
       &\geq \nu_0(\beta_1+\beta_2\nu_1)\\
       &\geq \sinh 1.
\end{align*}
Clearly, $\phi(r,s,t;z)$ lies outside the disk $\{\delta \in \mathbb{C}: |\delta -1|<\sinh 1\}$ which is enough to conclude that $\phi(r,s,t;z)\notin \Omega_s$. Therefore, $\phi\in \Psi[\Omega_s,1+\sin z]$ and the result follows as an application of Lemma \ref{10 millertheorem}.
\end{proof}

By taking $p(z)=zf'(z)/f(z)$ in Theorems 2.7-2.12, we obtain the following corollary with $S_f$ as given by \eqref {10s corollary21}: 
\begin{corollary}
Suppose $\beta_1$, $\beta_2>0$ and $f\in\mathcal{A}$. Then, $f\in \mathcal{S}^{*}_{s}$ if any of these conditions hold:
\begin{enumerate}[$(i)$]
   \item $S_{f}(z)\prec (1+Cz)/(1+Dz)$ and $\nu_0(\beta_1+\beta_2 \nu_1)(1-D^2)\geq (C-D)(1+|D|)$, where $-1< D<C\leq 1$.
\item $S_{f}(z)\prec 2/(1+e^{-z})$ and $\nu_0(\beta_1+\beta_2\nu_1)\geq r_0$, where $r_0\approx 0.546302$ is the positive root of the equation $r^2+2 \cot (1)r-1=0$.
\item $S_{f}(z)\prec z+\sqrt{1+z^2}$ and $\nu_0(\beta_1+\beta_2\nu_1)\geq \sqrt{2}$.
\item $S_{f}(z)\prec 1+ze^z$ and $\nu_0(\beta_1+\beta_2\nu_1)\geq e$.
\item $S_{f}(z)\prec 1+\sinh^{-1}z$ and $2\nu_0(\beta_1+\beta_2\nu_1)\geq \pi$.
\item $S_{f}(z)\prec e^z$ and $\nu_0(\beta_1+\beta_2\nu_1)\geq e-1$.
\item $S_{f}(z)\prec 1+\sin z$ and $\nu_0(\beta_1+\beta_2\nu_1)\geq \sinh 1$.
\end{enumerate}
\end{corollary}


Next, we determine the sufficient conditions on the positive constants $\beta_1$, $\beta_2$ and $\beta_3$ for various selections of $h(z)$, notably including the case when $h(z) = 1+\sin z$ itself. This study aims to establish third-order differential subordination implication, which is as follows:
\begin{equation*}
    1+\beta_1 z p'(z)+\beta_2 z^2p''(z)+\beta_3 z^3p'''(z)\prec h(z)\implies p(z)\prec 1+\sin z.
\end{equation*}
The derived conditions involve the variables $\nu_0$ and $\nu_1$, along with the constants $m$ and $k$, which are explicitly defined in \eqref{10s nu0}, \eqref{10s nu1} and Lemma \ref{lemmaformk}, respectively. We begin with the Theorem below by choosing $h(z)=\sqrt{1+z}$ as follows:

\begin{theorem}\label{10s thirdorderfirst}
Suppose $\beta_1$, $\beta_2$, $\beta_3>0$ and satisfy  $$\nu_0(\beta_1+\beta_2\nu_1+\beta_3(-m^2+3m(k-1)\nu_1))(\nu_0(\beta_1+\beta_2\nu_1+\beta_3(-m^2+3m(k-1)\nu_1))-2)\geq 1.$$  Let $p$ be analytic function in $\mathbb{D}$ with $p(0)=1$ and
\begin{equation*}
    1+\beta_1 z p'(z)+\beta_2 z^2p''(z)+\beta_3 z^3p'''(z)\prec \sqrt{1+z}.
\end{equation*}
Then $p(z)\prec 1+\sin z$.
\end{theorem}

\begin{proof}
Suppose $h(z)=\sqrt{1+z}$ for $z\in \mathbb{D}$. Then, $h(\mathbb{D})=\{w \in \mathbb{C}:|w ^2-1|<1\}=:\Omega$. Let $\phi:\mathbb{C}^4\times \mathbb{D}\rightarrow\mathbb{C}$ be defined as $\phi(r,s,t,u;z)=1+\beta_1 s+\beta_2 t+\beta_3 u$. We know that $\phi\in \Psi[\Omega,1+\sin z]$ provided $\phi(r,s,t,u;z)\notin \Omega$ for $z\in \mathbb{D}$. Consider
\begin{align*}
    |(\phi(r,s,t,u;z))^2-1|&=|(1+\beta_1 s+\beta_2 t+\beta_3 u)^2-1|\\
    &\geq |\beta_1 s+\beta_2 t+\beta_3 u|(|\beta_1 s+\beta_2 t+\beta_3 u|-2)\\
    &\geq |\beta_1 s|\RE\bigg(1+\frac{\beta_2}{\beta_1}\frac{t}{s}+\frac{\beta_3}{\beta_1}\frac{u}{s}\bigg)\bigg(|\beta_1 s|\RE\bigg(1+\frac{\beta_2}{\beta_1}\frac{t}{s}+\frac{\beta_3}{\beta_1}\frac{u}{s}\bigg)-2\bigg)\\
    &\geq M(M-2),
\end{align*}
where $M=m\beta_1 n_1(\theta)(1+\beta_2(m n_2(\theta)+m-1)/\beta_1+\beta_3(m^2n_3(\theta)+3m(k-1)n_2(\theta))/\beta_1)$. As $m\geq 1$ and $\RE(1+n_2(\theta))>0$, we get
\begin{align*}
  |(\phi(r,s,t,u;z))^2-1|&\geq \beta_1 n_1(\theta)\bigg(1+\frac{\beta_2}{\beta_1}n_2(\theta)+\frac{\beta_3}{\beta_1}(m^2n_3(\theta)+3m(k-1)n_2(\theta))\bigg)\bigg(\beta_1 n_1(\theta)\bigg(1+\\ 
  &\quad\frac{\beta_2}{\beta_1}n_2(\theta)+\frac{\beta_3}{\beta_1}(m^2n_3(\theta)+3m(k-1)n_2(\theta))\bigg)-2\bigg)\\
  &\geq\nu_0(\beta_1+\beta_2\nu_1+\beta_3(-m^2+3m(k-1)\nu_1))(\nu_0(\beta_1+\beta_2\nu_1+\beta_3(-m^2\\
  &\quad+3m(k-1)\nu_1))-2)\\
  &\geq 1.
\end{align*}
This implies that $\phi(r,s,t;z)\notin \Omega$ and thus $\phi\in \Psi[\Omega,1+\sin z]$. The result follows as a consequence of Lemma \ref{8 firsttheoremthirdorder}.
\end{proof}

By considering $p(z)=zf'(z)/f(z)$ in Theorem \ref{10s thirdorderfirst}, we have the following:
\begin{corollary}\label{10s corollary31}
Suppose $\beta_1$, $\beta_2$, $\beta_3>0$ and $f\in\mathcal{A}$. Let
\begin{align} \label{10s corollary31formula}
 \Theta_f(z)&:=1+\beta_1 S_1+(\beta_1+2\beta_2)(S_2-{S_1}^2)+(\beta_2+3\beta_3)(2{S_1}^3-3S_1S_2+3S_3)\nonumber\\
 &\quad+\beta_3(S_4-3{S_2}^2-{6S_1}^4-4S_1S_3+12{S_1}^2S_2).  
\end{align} 
where $S_1$, $S_2$ and $S_3$ are mentioned in \eqref{10s corollary21coefficients} and $S_4:=z^4f^{(iv)}(z)/f(z)$.
Then, $f(z)\in \mathcal{S}^{*}_{s}$ provided $\Theta_{f}(z)\prec \sqrt{1+z}$ and $\nu_0(\beta_1+\beta_2\nu_1+\beta_3(-m^2+3m(k-1)\nu_1))-2\nu_0(\beta_1+\beta_2\nu_1+\beta_3(-m^2+3m(k-1)\nu_1))\geq 1$.
\end{corollary}

\begin{theorem}
Suppose $\beta_1$, $\beta_2$, $\beta_3>0$ and $-1< D<C\leq 1$ with $\nu_0(\beta_1+\beta_2\nu_1-m^2\beta_3+3m\beta_3(k-1)\nu_1)(1-D^2)\geq (1+|D|)(C-D)$. Let $p$ be analytic function in $\mathbb{D}$ with $p(0)=1$ and
\begin{equation*}
    1+\beta_1 z p'(z)+\beta z^2p''(z)+\beta_3 z^3p'''(z)\prec \frac{1+Cz}{1+Dz}.
\end{equation*}
Then $p(z)\prec 1+\sin z$.
\end{theorem}

\begin{proof}
Let $h(z)=(1+Cz)/(1+Dz)$ for $z\in \mathbb{D}$. Then, $h(\mathbb{D})=\{w \in \mathbb{C}:|w -(1-CD)/(1-D^2)|<(C-D)/(1-D^2)\}=:\Omega$. Let $\phi:\mathbb{C}^4\times \mathbb{D}\rightarrow\mathbb{C}$ be defined as $\phi(r,s,t,u;z)=1+\beta_1 s+\beta_2 t+\beta_3 u$. For $\phi\in\Psi[\Omega,1+\sin z]$, it is needed that $\phi(r,s,t,u;z)\notin \Omega$ for $z\in \mathbb{D}$. Consider
\begin{align*}
   \bigg| \phi(r,s,t,u;z)-\frac{1-CD}{1-D^2}\bigg|&=\bigg|1+\beta_1 s+\beta_2 t+\beta_3 u-\frac{1-CD}{1-D^2}\bigg|\\
   &\geq |\beta_1 s+\beta_2 t+\beta_3 u|-\frac{|D|(C-D)}{1-D^2}\\
   &\geq |\beta_1 s|\RE\bigg(1+\frac{\beta_2}{\beta_1}\frac{t}{s}+\frac{\beta_3}{\beta_1}\frac{u}{s}\bigg)-\frac{|D|(C-D)}{1-D^2}\\
   &\geq m\beta_1 n_1(\theta)\bigg(1+\frac{\beta_2}{\beta_1}(mn_2(\theta)+m-1)+\frac{\beta_3}{\beta_1}(m^2n_3(\theta)+3m(k-1)n_2(\theta))\bigg)\\
   &\quad-\frac{|D|(C-D)}{1-D^2}.
\end{align*}
Proceeding as in the proof of Theorem \ref{10s thirdorderfirst} and the fact that $m\geq 1$, we obtain
\begin{align*}
    \bigg| \phi(r,s,t,u;z)-\frac{1-CD}{1-D^2}\bigg|&\geq 
     \nu_0\bigg(\beta_1+\beta_2\nu_1+\beta_3(-m^2+3m(k-1)\nu_1)\bigg)-\frac{|D|(C-D)}{1-D^2}\\
    &\geq \frac{C-D}{1-D^2}.
\end{align*}
From the above calculations, we conclude that $\phi\in \Psi[\Omega,1+\sin z]$ and hence the result follows as a consequence of Lemma \ref{8 firsttheoremthirdorder}.
\end{proof}

\begin{theorem}\label{10s sigmoid3}
Suppose $\beta_1$, $\beta_2$, $\beta_3>0$ and $\nu_0(\beta_1+\beta_2\nu_1-m^2\beta_3+3m\beta_3(k-1)\nu_1)\geq r_0$, where $r_0\approx 0.546302$ is the positive root of the equation $r^2+2 \cot (1)r-1=0$. Let $p$ be analytic function in $\mathbb{D}$ with $p(0)=1$ and
\begin{equation*}
1+\beta_1 zp'(z)+\beta_2 z^2p''(z)+\beta_3 z^3p'''(z)\prec \frac{2}{1+e^{-z}}.       
\end{equation*}
Then $p(z)\prec 1+\sin z$.
\end{theorem}

\begin{proof}
Assume $h(z)=2/(1+e^{-z})$ for $z\in \mathbb{D}$. Then, $h(\mathbb{D})=\{w \in \mathbb{C}:|\log (w /(2-w ))|<1\}=:\Omega$. Let $\phi:\mathbb{C}^4\times \mathbb{D}\rightarrow \mathbb{C}$ be defined as $\phi(r,s,t,u;z)=1+\beta_1 s+\beta_2 t+\beta_3 u$. For $\phi\in \Psi[\Omega, 1+\sin z]$, we must have $\phi(r,s,t,u;z)\notin \Omega$. Consider
\begin{align*}
    |\beta_1 s+\beta_2 t+\beta_3 u|&=\beta_1 |s|\bigg|1+\frac{\beta_2}{\beta_1}\frac{t}{s}+\frac{\beta_3}{\beta_1}\frac{u}{s}\bigg|\\
    &\geq \beta_1 |s|\RE\bigg(1+\frac{\beta_2}{\beta_1}\frac{t}{s}+\frac{\beta_3}{\beta_1}\frac{u}{s}\bigg)\\
    &\geq m \beta_1 n_1(\theta)\bigg(1+\frac{\beta_2}{\beta_1}(mn_2(\theta)+m-1)+\frac{\beta_3}{\beta_1}(m^2n_3(\theta)+3m(k-1)n_2(\theta))\bigg).
\end{align*}
As $m\geq 1$ and $\RE (1+n_2(\theta))>0$, we have

\begin{align}\label{10s 73}
    |\beta_1 s+\beta_2 t+\beta_3 u|&\geq n_1(\theta)\bigg(\beta_1+\beta_2 n_2(\theta)+\beta_3 (m^2n_3(\theta)+3m(k-1)n_2(\theta))\bigg)\nonumber\\
    &\geq \nu_0(\beta_1+\beta_2\nu_1-m^2\beta_3+3m\beta_3(k-1)\nu_1)\nonumber\\
    &\geq r_0.
\end{align}
Now, we consider
\begin{equation*}
    \bigg|\log \bigg(\frac{\phi(r,s,t,u;z)}{2-\phi(r,s,t,u;z)}\bigg)\bigg|=\bigg|\log\bigg(\frac{1+\beta_1 s+\beta_2 t+\beta_3 u}{1-(\beta_1 s+\beta_2 t+\beta_3 u)}\bigg)\bigg|.
\end{equation*}
In view of Lemma \ref{prilemma42} and \eqref{10s 73}, we have
\begin{equation*}
    \bigg|\log \bigg(\frac{1+\beta_1 s+\beta_2 t+\beta_3 u}{1-(\beta_1 s+\beta_2 t+\beta_3 u)}\bigg)\bigg|\geq 1,
\end{equation*}
which implies that $\phi\in \Psi[\Omega,1+\sin z]$. Thus, The result follows as an application of Lemma \ref{8 firsttheoremthirdorder}.
\end{proof}

\begin{theorem}
Suppose $\beta_1$, $\beta_2$, $\beta_3>0$ and $\nu_0(\beta_1+\beta_2\nu_1-m^2\beta_3+3m\beta_3 (k-1)\nu_1)\geq \sqrt{2}$. Let $p$ be analytic function in $\mathbb{D}$ with $p(0)=1$ and
\begin{equation*}
        1+\beta_1 zp'(z)+\beta_2 z^2p''(z)+\beta_3 z^3p'''(z)\prec z+\sqrt{1+z^2}.
\end{equation*}
Then $p(z)\prec 1+\sin z$.
\end{theorem}

\begin{proof}
Take $h(z)=z+\sqrt{1+z^2}$ for $z\in \mathbb{D}$. Then, $h(\mathbb{D})=\{w \in \mathbb{C}:|w ^2-1|<2|w |\}=:\Omega$. Let $\phi:\mathbb{C}^4\times \mathbb{D}\rightarrow \mathbb{C}$ be defined as $\phi(r,s,t,u;z)=1+\beta_1 s+\beta_2 t+\beta_3 u$. For $\phi\in \Psi[\Omega, 1+\sin z]$, we must have $\phi(r,s,t,u;z)\notin \Omega$. From the geometry of $z+\sqrt{1+z^2}$ (see Fig. \ref{crescent10}), we note that $\Omega$ is constructed by two circles
\begin{equation*}
        C_1:|z-1|=\sqrt{2}\quad\text{and}\quad C_2:|z+1|=\sqrt{2}.
\end{equation*}
It is obvious that $\Omega$ contains the disk enclosed by $C_1$ and excludes the portion of the disk enclosed by $C_1\cap C_2$. We have
\begin{equation*}
        |\phi(r,s,t,u;z)-1|=|\beta_1 s+\beta_2 t+\beta_3 u|.
\end{equation*}
Analogous to Theorem \ref{10s sigmoid3}, we have
\begin{align*}
       |\beta_1 s+\beta_2 t+\beta_3 u|&\geq n_1(\theta)\bigg(\beta_1+\beta_2 n_2(\theta)+\beta_3(m^2n_3(\theta)+3m(k-1)n_2(\theta))\bigg)\\
       &\geq \nu_0(\beta_1+\beta_2\nu_1-m^2\beta_3+3m\beta_3 (k-1)\nu_1)\\
       &\geq \sqrt{2}.
\end{align*}
Thus, we can say that $\phi(r,s,t,u;z)$ lies outside the circle $C_1$ which is sufficient to deduce that $\phi(r,s,t,u;z)\notin \Omega$. Therefore, $\phi\in \Psi[\Omega,1+\sin z]$ and thus the result follows as a consequence of Lemma \ref{8 firsttheoremthirdorder}.
\end{proof}


\begin{theorem}
Suppose $\beta_1$, $\beta_2$, $\beta_3>0$ and $\nu_0(\beta_1+\beta_2\nu_1-m^2\beta_3+3m\beta_3(k-1)\nu_1)\geq e$. Let $p$ be analytic function in $\mathbb{D}$ with $p(0)=1$ and
\begin{equation*}
        1+\beta_1 zp'(z)+\beta_2 z^2p''(z)+\beta_3 z^3p'''(z)\prec 1+ze^z.
\end{equation*}
Then $p(z)\prec 1+\sin z$.
\end{theorem}

\begin{proof}
Let $h(z)=1+ze^z$ for $z\in \mathbb{D}$, then $\Omega:=h(\mathbb{D})$. Let $\phi:\mathbb{C}^4\times \mathbb{D}\rightarrow \mathbb{C}$ be defined as $\phi(r,s,t,u;z)=1+\beta_1 s+\beta_2 t+\beta_3 u$. For $\phi\in \Psi[\Omega, 1+\sin z]$, we must have $\phi(r,s,t,u;z)\notin \Omega$. From \cite[Lemma 3.3]{kumar-ganganiaCardioid-2021}, we note that the smallest disk containing $\Omega$ is $\{\delta \in\mathbb{C}: |\delta -1|<e\}$. Also,
\begin{equation*}
        |\phi(r,s,t,u;z)-1|=|\beta_1 s+\beta_2 t+\beta_3 u|.
\end{equation*}
On the similar lines of the proof of Theorem \ref{10s sigmoid3}, we have
\begin{align*}
       |\beta_1 s+\beta_2 t+\beta_3 u|&\geq n_1(\theta)\bigg(\beta_1+\beta_2 n_2(\theta)+\beta_3(m^2n_3(\theta)+3m(k-1)n_2(\theta))\bigg)\\
       &\geq \nu_0(\beta_1+\beta_2\nu_1-m^2\beta_3+3m\beta_3(k-1)\nu_1)\\
       &\geq e.
\end{align*}
Clearly, $\phi(r,s,t,u;z)$ lies outside the disk $\{\delta \in \mathbb{C}: |\delta -1|<e\}$ which is enough to conclude that $\phi(r,s,t,u;z)\notin \Omega$. Therefore, $\phi\in \Psi[\Omega,1+\sin z]$ and hence, the result follows as an application of Lemma \ref{8 firsttheoremthirdorder}.
\end{proof}

\begin{theorem}
Suppose $\beta_1$, $\beta_2$, $\beta_3>0$ and $2\nu_0(\beta_1+\beta_2\nu_1 -m^2\beta_3+3m\beta_3(k-1)\nu_1)\geq \pi$.  Let $p$ be analytic function in $\mathbb{D}$ with $p(0)=1$ and
\begin{equation*}
        1+\beta_1 zp'(z)+\beta_2 z^2p''(z)+\beta_3 z^3p'''(z)\prec 1+\sinh^{-1} z.
\end{equation*}
Then $p(z)\prec 1+\sin z$.
\end{theorem}

\begin{proof}
Assume $h(z)=1+\sinh^{-1} z$ for $z\in \mathbb{D}$. Then, $h(\mathbb{D})=\{w \in\mathbb{C}:|\sinh (w -1)|<1\}=:\Omega_{\rho}$. Let $\phi:\mathbb{C}^4\times \mathbb{D}\rightarrow \mathbb{C}$ be defined as $\phi(r,s,t,u;z)=1+\beta_1 s+\beta_2 t+\beta_3 u$. For $\phi\in \Psi[\Omega_\rho, 1+\sin z]$, we must have $\phi(r,s,t,u;z)\notin \Omega_\rho$. Through \cite[Remark 2.7]{kush}, we note that the smallest disk containing $\Omega_\rho$ is $\{\delta \in\mathbb{C}: |\delta -1|<\pi/2\}$. Thus
\begin{equation*}
        |\phi(r,s,t,u;z)-1|=|\beta_1 s+\beta_2 t+\beta_3 u|.
\end{equation*}
Proceeding on the similar lines of Theorem \ref{10s sigmoid3}, we have
\begin{align*}
       |\beta_1 s+\beta_2 t+\beta_3 u|&\geq n_1(\theta)\bigg(\beta_1+\beta_2 n_2(\theta)+\beta_3(m^2n_3(\theta)+3m(k-1)n_2(\theta))\bigg)\\
       &\geq \nu_0(\beta_1+\beta_2\nu_1-m^2\beta_3+3m\beta_3(k-1)\nu_1)\\
       &\geq \frac{\pi}{2}.
\end{align*}
Clearly, $\phi(r,s,t,u;z)$ lies outside the disk $\{\delta \in \mathbb{C}: |\delta -1|<\pi/2\}$ which suffices to prove that $\phi(r,s,t,u;z)\notin \Omega_\rho$. Therefore, $\phi\in \Psi[\Omega_\rho,1+\sin z]$ and thus the result follows as an application of Lemma \ref{8 firsttheoremthirdorder}.
\end{proof}


\begin{theorem}
Suppose $\beta_1$, $\beta_2$, $\beta_3>0$ and $\nu_0(\beta_1+\beta_2\nu_1-m^2\beta_3+3m\beta_3(k-1)\nu_1)\geq e-1$.  Let $p$ be analytic function in $\mathbb{D}$ with $p(0)=1$ and
\begin{equation*}
        1+\beta_1 zp'(z)+\beta_2 z^2p''(z)+\beta_3 z^3p'''(z)\prec e^z.
\end{equation*}
Then $p(z)\prec 1+\sin z$.
\end{theorem}

\begin{proof}
Take $h(z)=e^z$ for $z\in \mathbb{D}$. Then, $h(\mathbb{D})=\{w \in\mathbb{C}:|\log w |<1\}=:\Omega$. Let $\phi:\mathbb{C}^4\times \mathbb{D}\rightarrow \mathbb{C}$ be defined as $\phi(r,s,t,u;z)=1+\beta_1 s+\beta_2 t+\beta_3 u$. For $\phi\in \Psi[\Omega, 1+\sin z]$, we must have $\phi(r,s,t,u;z)\notin \Omega$. So,
\begin{equation*}
        |\phi(r,s,t,u;z)-1|=|\beta_1 s+\beta_2 t+\beta_3 u|.
\end{equation*}
Proceeding on the similar lines of proof of Theorem \ref{10s sigmoid3}, we get
\begin{align}\label{10s 83}
       |\beta_1 s+\beta_2 t+\beta_3 u|&\geq n_1(\theta)\bigg(\beta_1+\beta_2 n_2(\theta)+\beta_3(m^2n_3(\theta)+3m(k-1)n_2(\theta))\bigg)\nonumber\\
       &\geq \nu_0(\beta_1+\beta_2\nu_1-m^2\beta_3+3m\beta_3(k-1)\nu_1)\nonumber\\
       &\geq e-1.
\end{align}
Next, we consider
\begin{equation*}
     |\log (\phi(r,s,t,u;z)|=|\log (1+\beta_1 s+\beta_2 t+\beta_3 u)|.
\end{equation*}
Through Lemma \ref{prilemma4} and \eqref{10s 83}, we have
\begin{equation*}
   |\log (1+\beta_1 s+\beta_2 t+\beta_3 u)|\geq 1, 
\end{equation*}
 which implies that $\phi(r,s,t,u;z)\notin \Omega$. Therefore, $\phi\in \Psi[\Omega,1+\sin z]$ and the result follows as a consequence of Lemma \ref{8 firsttheoremthirdorder}. 
\end{proof}

\begin{theorem}
Suppose $\beta_1$, $\beta_2$, $\beta_3>0$ and $\nu_0(\beta_1+\beta_2\nu_1-m^2\beta_3+3m\beta_3(k-1)\nu_1)\geq \sinh 1$.  Let $p$ be analytic function in $\mathbb{D}$ with $p(0)=1$ and
\begin{equation*}
        1+\beta_1 zp'(z)+\beta_2 z^2p''(z)+\beta_3 z^3p'''(z)\prec 1+\sin z.
\end{equation*}
Then $p(z)\prec 1+\sin z$.
\end{theorem}

\begin{proof}
Take $h(z)=1+\sin z$ for $z\in \mathbb{D}$. Then, $h(\mathbb{D})=\{w \in \mathbb{C}:|\arcsin (w -1)|<1\}=:\Omega_s$. Let $\phi:\mathbb{C}^4\times \mathbb{D}\rightarrow \mathbb{C}$ be defined as $\phi(r,s,t,u;z)=1+\beta_1 s+\beta_2 t+\beta_3 u$. For $\phi\in \Psi[\Omega_s, 1+\sin z]$, we must have $\phi(r,s,t,u;z)\notin \Omega_s$. Through \cite[Lemma 3.3]{chosine}, we note that the smallest disk containing $\Omega_s$ is $\{\delta \in \mathbb{C}: |\delta -1|<\sinh 1\}$. So,
\begin{equation*}
        |\phi(r,s,t,u;z)-1|=|\beta_1 s+\beta_2 t+\beta_3 u|.
\end{equation*}
Similar to the proof of Theorem \ref{10s sigmoid3}, we have
\begin{align*}
       |\beta_1 s+\beta_2 t+\beta_3 u|&\geq n_1(\theta)\bigg(\beta_1+\beta_2 n_2(\theta)+\beta_3(m^2n_3(\theta)+3m(k-1)n_2(\theta))\bigg)\\
       &\geq \nu_0(\beta_1+\beta_2\nu_1 -m^2\beta_3+3m\beta_3(k-1)\nu_1)\\
       &\geq \sinh 1.
\end{align*}
Clearly, $\phi(r,s,t,u;z)$ lies outside the disk $\{\delta \in \mathbb{C}: |\delta-1|<\sinh 1\}$ which is enough to conclude that $\phi(r,s,t,u;z)\notin \Omega_s$. Therefore, $\phi\in \Psi[\Omega_s,1+\sin z]$ and thus the result holds through Lemma \ref{8 firsttheoremthirdorder}.
\end{proof}

Now, we summarize this section by providing a combined conclusion. If $p(z)=zf'(z)/f(z)$ is considered in Theorems 2.18-2.23 as in Corollary \ref{10s corollary31}. We deduce the following
\begin{corollary}
Suppose $\beta_1$, $\beta_2$, $\beta_3>0$ and $f\in\mathcal{A}$. Then, $f\in \mathcal{S}^{*}_{s}$ if any of the conditions hold:
\begin{enumerate}[$(i)$]
\item $\Theta_{f}(z)\prec (1+Cz)/(1+Dz)$ and $\nu_0(\beta_1+\beta_2\nu_1-m^2\beta_3+3m\beta_3(k-1)\nu_1)(1-D^2)\geq (1+|D|)(C-D)$, where $-1< D<C\leq 1$
\item $\Theta_{f}(z)\prec 2/(1+e^{-z})$ and $\nu_0(\beta_1+\beta_2\nu_1-m^2\beta_3+3m\beta_3(k-1)\nu_1)\geq r_0$, where $r_0\approx 0.546302$ is the positive root of the equation $r^2+2 \cot (1)r-1=0$.
\item $\Theta_{f}(z)\prec z+\sqrt{1+z^2}$ and $\nu_0(\beta_1+\beta_2\nu_1-m^2\beta_3+3m\beta_3 (k-1)\nu_1)\geq \sqrt{2}$.
\item $\Theta_{f}(z)\prec 1+ze^z$ and $\nu_0(\beta_1+\beta_2\nu_1-m^2\beta_3+3m\beta_3(k-1)\nu_1)\geq e$.
\item $\Theta_{f}(z)\prec 1+\sinh^{-1}z$ and $2\nu_0(\beta_1+\beta_2\nu_1 -m^2\beta_3+3m\beta_3(k-1)\nu_1)\geq \pi$.
\item $\Theta_{f}(z)\prec e^z$ and $\nu_0(\beta_1+\beta_2\nu_1-m^2\beta_3+3m\beta_3(k-1)\nu_1)\geq e-1$.
\item $\Theta_{f}(z)\prec 1+\sin z$ and $\nu_0(\beta_1+\beta_2\nu_1-m^2\beta_3+3m\beta_3(k-1)\nu_1)\geq \sinh 1$.
\end{enumerate}
\end{corollary}

\section{The Class $\mathcal{S}^{*}_{\rho}$}

\noindent Let us consider the function $q(z):=1+\sinh^{-1}z$ and define the admissibility class $\Psi[\Omega,q]$, where $\Omega\subset \mathbb{C}$. We know that $q$ is analytic and univalent on $\overline{\mathbb{D}}$, $q(0)=1$ and it maps $\mathbb{D}$ onto the domain $\Omega_{\rho}:=\{w \in \mathbb{C}:|\sinh (w -1)|<1\}$. Since $\mathbb{E}(q)=\phi$, for $\zeta\in \partial \mathbb{D}\setminus \mathbb{E}(q)$ if and only if $\zeta=e^{i\theta}$ for $\theta\in [0,2\pi]$. Now, consider the function
\begin{equation}
    |q'(\zeta)|=\frac{1}{\sqrt{2}(\cos \theta)^{1/2}}=:n_4(\theta),\label{10 1}
\end{equation}
which has a minimum value of $1/\sqrt{2}$. It is evident that $\min |q'(\zeta)|>0$, which implies that $q\in Q(1)$, and consequently, the admissibility class $\Psi[\Omega,q]$ is well-defined. While considering $|\zeta|=1$, we observe that $q(\zeta)\in q(\partial \mathbb{D})=\partial \Omega_{\rho} = \{w  \in \mathbb{C}: |\sinh(w -1)| = 1\}$. Consequently, we have $|\log q(\zeta)| = 1$ and $\sinh (q(\zeta)-1) = e^{i\theta}$ ($\theta\in[0,2\pi]$), implies that $q(\zeta) = 1+\sinh^{-1}\zeta$. Moreover, $\zeta q'(\zeta)=e^{i\theta}/(1+e^{2i\theta})^{1/2}$ and 
\begin{equation}
    \frac{\zeta q''(\zeta)}{q'(\zeta)}=-\frac{e^{2i\theta}}{1+e^{2i\theta}}.\label{10 2}
\end{equation}
By comparing the real parts on both sides of \eqref{10 2}, we obtain
\begin{equation}
    \RE\bigg(\frac{\zeta q''(\zeta)}{q'(\zeta)}\bigg)=-\frac{1}{2}=:n_5(\theta).\label{10 3}
\end{equation}
The function $n_5(\theta)$ defined in \eqref{10 3} has minimum value $-1/2$ being a constant function. Moreover, the class $\Psi[\Omega, 1+\sinh^{-1}z]$ is precisely defined as the class of all functions $\phi: \mathbb{C}^3 \times \mathbb{D} \rightarrow \mathbb{C}$ that satisfy the following conditions:

\begin{equation*}
    \phi(r,s,t;z)\notin \Omega\quad \text{for}\quad z\in \mathbb{D},\quad \theta\in[0,2\pi] \quad\text{and}\quad m\geq 1,
\end{equation*}
whenever
\begin{equation}
    r=q(\zeta)=1+\sinh^{-1}(e^{i\theta});\quad s=m\zeta q'(\zeta)=\frac{me^{i\theta}}{\sqrt{1+e^{2i\theta}}};\quad \RE\bigg(1+\frac{t}{s}\bigg)\geq m(1+n_5(\theta)). \label{10 4}
\end{equation}
If $\phi:\mathbb{C}^2\times \mathbb{D}\rightarrow \mathbb{C}$, then the admissibility condition \eqref{10 4} becomes 
\begin{equation*}
\phi\bigg(1+\sinh^{-1}(e^{i\theta}),\frac{me^{i\theta}}{\sqrt{1+e^{2i\theta}}};z\bigg)\notin \Omega\quad (z\in \mathbb{D}, \theta\in[0,2\pi], m\geq 1).
\end{equation*}

Now, we discuss the implications of the following second-order differential subordination 
\begin{equation*}
    1+\beta_1 zp'(z)+\beta_2 z^2p''(z)\prec h(z) \implies p(z)\prec 1+\sinh^{-1}z,
\end{equation*}
by finding sufficient conditions on the parameters $\beta_1$ and $\beta_2$ for various choices of $h(z)$ including $1+\sinh^{-1}z$ itself. We begin by considering $h(z)=\sqrt{1+z}$ in the following manner:

\begin{theorem}\label{10 secondorder1}
Suppose $\beta_1$, $\beta_2>0$ and $(2{\beta_1}- \beta_2)(2\beta_1-\beta_2-4\sqrt{2})\geq 8$. Let $p$ be analytic function in $\mathbb{D}$ with $p(0)=1$ and
\begin{equation*}
      1+\beta_1 zp'(z)+\beta_2 z^2 p''(z)\prec \sqrt{1+z}.
\end{equation*}
Then $p(z)\prec 1+\sinh^{-1}z$.
\end{theorem}

\begin{proof}
Consider $h(z)=\sqrt{1+z}$ for $z\in \mathbb{D}$. Then, $h(\mathbb{D})=\{w \in \mathbb{C}:|w ^2-1|<1\}=:\Omega$. Let $\phi:\mathbb{C}^3\times \mathbb{D}\rightarrow \mathbb{C}$ be defined as $\phi(r,s,t;z)=1+\beta_1 s+\beta_2 t$. It is clear that $\phi\in \Psi[\Omega,1+\sinh^{-1} z]$ provided $\phi(r,s,t,;z)\notin \Omega$ for $z\in \mathbb{D}$. Note that
\begin{align*}
        |(\phi(r,s,t;z))^2-1|&=|(1+\beta_1 s+\beta_2 t)^2-1|\\
        &\geq |\beta_1 s+\beta_2 t|(|\beta_1 s+\beta_2 t|-2)\\
        &\geq |\beta_1 s|\RE\bigg(1+\frac{\beta_2}{\beta_1}\frac{t}{s}\bigg)\bigg(|\beta_1 s|\RE\bigg(1+\frac{\beta_2}{\beta_1}\frac{t}{s}\bigg)-2\bigg)\\
        &\geq m\beta_1 n_4(\theta)\RE\bigg(1+\frac{\beta_2}{\beta_1}(mn_5(\theta)+m-1)\bigg)\bigg(m\beta_1 n_4(\theta)\RE\bigg(1+\\
        &\quad \frac{\beta_2}{\beta_1}(mn_5(\theta)+m-1)\bigg)-2\bigg).
\end{align*}
Here $n_4(\theta)$ and $n_5(\theta)$ are given in \eqref{10 1} and \eqref{10 3}. Since $m\geq 1$, therefore
\begin{align*}
       |(\phi(r,s,t;z))^2-1|&\geq \beta_1 n_4(\theta)\RE\bigg(1+\frac{\beta_2}{\beta_1}n_5(\theta)\bigg)\bigg(\beta_1 n_4(\theta)\RE\bigg(1+\frac{\beta_2}{\beta_1}n_5(\theta)\bigg)-2\bigg)\\
       &\geq \frac{1}{\sqrt{2}}\bigg(\beta_1-\frac{\beta_2}{2}\bigg) \bigg(\frac{1}{\sqrt{2}}\bigg(\beta_1-\frac{\beta_2}{2}\bigg)-2\bigg)\\
       &=\frac{(2\beta_1-\beta_2)(2\beta_1-\beta_2-4\sqrt{2})}{8}\\
       &\geq 1.
\end{align*}
Therefore, $\phi(r,s,t;z)\notin \Omega$ and hence $\phi\in \Psi[\Omega,1+\sinh^{-1} z]$. Thus, result follows as an application of Lemma \ref{10 millertheorem}.
\end{proof}

\begin{theorem}
Suppose $\beta_1$, $\beta_2>0$ and $-1< D<C\leq 1$ with $(2\beta_1 -\beta_2)(1-D^2)\geq 2\sqrt{2}(C-D)(1+|D|)$. Let $p$ be analytic function in $\mathbb{D}$ with $p(0)=1$ and
\begin{equation*}
        1+\beta_1 zp'(z)+\beta_2 z^2p''(z)\prec \frac{1+Cz}{1+Dz}.
\end{equation*}
Then $p(z)\prec 1+\sinh^{-1}z$.
\end{theorem}

\begin{proof}
Assume $h(z)=(1+Cz)/(1+Dz)$ for $z\in \mathbb{D}$. Then, $h(\mathbb{D})=\{w \in \mathbb{C}:|w -(1-CD)/(1-D^2)|<(C-D)/(1-D^2)\}=:\Omega$. 
Let $\phi:\mathbb{C}^3\times \mathbb{D}\rightarrow \mathbb{C}$ be defined as $\phi(r,s,t;z)=1+\beta_1 s+\beta_2 t$. It is clear that $\phi\in \Psi[\Omega,1+\sinh^{-1} z]$ if $\phi(r,s,t;z)\notin \Omega$ for $z\in \mathbb{D}$. Consider
\begin{align*}
        \bigg|\phi(r,s,t;z)-\frac{1-CD}{1-D^2}\bigg|&=\bigg|1+\beta_1 s+\beta t-\frac{1-CD}{1-D^2}\bigg|\\
        &\geq |\beta_1 s+\beta_2 t|-\frac{|D|(C-D)}{1-D^2}\\
        &\geq |\beta_1 s|\RE\bigg(1+\frac{\beta_2}{\beta_1}\frac{t}{s}\bigg)-\frac{|D|(C-D)}{1-D^2}\\
        &\geq m\beta_1 n_4(\theta)\RE\bigg(1+\frac{\beta_2}{\beta_1}(mn_5(\theta)+m-1)\bigg)-\frac{|D|(C-D)}{1-D^2}.
\end{align*}
Using the fact that $m\geq 1$ and proceeding on the similar lines of the proof of Theorem \ref{10 secondorder1}, we have
\begin{align*}
   \bigg|\phi(r,s,t;z)-\frac{1-CD}{1-D^2}\bigg|
   &\geq\frac{1}{\sqrt{2}}\bigg(\beta_1-\frac{\beta_2}{2}\bigg)-\frac{|D|(C-D)}{1-D^2}\\
   &\geq \frac{C-D}{1-D^2}.
\end{align*}
Thus, $\phi\in \Psi[\Omega,1+\sinh^{-1} z]$ and the result follows as an application of Lemma \ref{10 millertheorem}.
\end{proof}

\begin{theorem}\label{10 thmsigmoid}
Suppose $\beta_1$, $\beta_2>0$ and $2\beta_1-\beta_2\geq 2\sqrt{2} r_0$, where $r_0\approx 0.546302$ is the positive root of the equation $r^2+2 \cot (1)r-1=0$. Let $p$ be analytic function in $\mathbb{D}$ with $p(0)=1$ and
\begin{equation*}
1+\beta_1 zp'(z)+\beta_2 z^2p''(z)\prec \frac{2}{1+e^{-z}}.        
\end{equation*}
Then $p(z)\prec 1+\sinh^{-1}z$.
\end{theorem}

\begin{proof}
Suppose $h(z)=2/(1+e^{-z})$ for $z\in \mathbb{D}$. Then, $h(\mathbb{D})=\{w \in \mathbb{C}:|\log (w /(2-w ))|<1\}=:\Omega$. Let $\phi:\mathbb{C}^3\times \mathbb{D}\rightarrow \mathbb{C}$ be defined as $\phi(r,s,t;z)=1+\beta_1 s+\beta_2 t$. We know that $\phi\in \Psi[\Omega, 1+\sinh^{-1} z]$ only when $\phi(r,s,t;z)\notin \Omega$. First we consider,
\begin{align*}
    |\beta_1 s+\beta_2 t|&=\beta_1 |s|\bigg|1+\frac{\beta_2}{\beta_1}\frac{t}{s}\bigg|\\
    &\geq \beta_1 |s|\RE\bigg(1+\frac{\beta_2}{\beta_1}\frac{t}{s}\bigg)\\
    &\geq m \beta_1 n_4(\theta)\bigg(1+\frac{\beta_2}{\beta_1}(mn_5(\theta)+m-1)\bigg).
\end{align*}
As $m\geq 1$, we obtain
\begin{align}\label{10 7}
    |\beta_1 s+\beta_2 t|&\geq n_4(\theta)(\beta_1+\beta_2 n_5(\theta))\nonumber\\
    &\geq \frac{1}{\sqrt{2}}\bigg(\beta_1-\frac{1}{2}\beta_2\bigg)\nonumber\\
    &\geq r_0.
\end{align}
Now, we consider
\begin{equation*}
    \bigg|\log \bigg(\frac{\phi(r,s,t;z)}{2-\phi(r,s,t;z)}\bigg)\bigg|=\bigg|\log\bigg(\frac{1+\beta_1 s+\beta_2 t}{1-(\beta_1 s+\beta_2 t)}\bigg)\bigg|.
\end{equation*}
Through Lemma \ref{prilemma42} and \eqref{10 7}, we have
\begin{equation*}
    \bigg|\log \bigg(\frac{1+\beta_1 s+\beta_2 t}{1-(\beta_1 s+\beta_2 t)}\bigg)\bigg|\geq 1,
\end{equation*}
which implies that $\phi\in \Psi[\Omega,1+\sinh^{-1} z]$. Therefore, result follows as an application of Lemma \ref{10 millertheorem}.
\end{proof}

\begin{theorem}
Suppose $\beta_1$, $\beta_2>0$ and $2\beta_1-\beta_2\geq 4$. Let $p$ be analytic function in $\mathbb{D}$ with $p(0)=1$ and
\begin{equation*}
        1+\beta_1 zp'(z)+\beta_2 z^2p''(z)\prec z+\sqrt{1+z^2}.
\end{equation*}
Then $p(z)\prec 1+\sinh^{-1}z$.
\end{theorem}

\begin{proof}
Take $h(z)=z+\sqrt{1+z^2}$ for $z\in \mathbb{D}$. Then, $h(\mathbb{D})=\{w \in \mathbb{C}:|w^2-1|<2|w|\}=:\Omega$. Let $\phi:\mathbb{C}^3\times \mathbb{D}\rightarrow \mathbb{C}$ be defined as $\phi(r,s,t;z)=1+\beta_1 s+\beta_2 t$. For $\phi\in \Psi[\Omega, 1+\sinh^{-1} z]$, we must have $\phi(r,s,t;z)\notin \Omega$. From the geometry of $z+\sqrt{1+z^2}$ (see Fig. \ref{crescent10}), we note that $\Omega$ is constructed by the circles
\begin{equation*}
        C_1:|z-1|=\sqrt{2}\quad\text{and}\quad C_2:|z+1|=\sqrt{2}.
\end{equation*}

It is obvious that $\Omega$ contains the disk enclosed by $C_1$ and excludes the portion of the disk enclosed by $C_1\cap C_2$. We have
\begin{equation*}
        |\phi(r,s,t;z)-1|=|\beta_1 s+\beta_2 t|.
\end{equation*}
Similar to the proof of Theorem \ref{10 thmsigmoid}, we have
\begin{align*}
       |\beta_1 s+\beta_2 t|&\geq n_4(\theta)(\beta_1+\beta_2 n_5(\theta))\\
       &\geq \frac{1}{\sqrt{2}}\bigg(\beta_1-\frac{1}{2}\beta_2 \bigg)\\
       &\geq \sqrt{2}.
\end{align*}
The fact that $\phi(r,s,t;z)$ lies outside the circle $C_1$ suffices us to deduce that $\phi(r,s,t;z)\notin \Omega$. 
Consequently, $\phi\in \Psi[\Omega,1+\sinh^{-1}z]$ and thus the result follows as an application of Lemma \ref{10 millertheorem}.
\end{proof}

\begin{theorem}
Suppose $\beta_1$, $\beta_2>0$ and $2\beta_1-\beta_2\geq 2\sqrt{2}\sinh 1$. Let $p$ be analytic function in $\mathbb{D}$ with $p(0)=1$ and
\begin{equation*}
        1+\beta_1 zp'(z)+\beta_2 z^2p''(z)\prec 1+\sin z.
\end{equation*}
Then $p(z)\prec 1+\sinh^{-1}z$.
\end{theorem}

\begin{proof}
Consider $h(z)=1+\sin z$ for $z\in \mathbb{D}$. Then, $h(\mathbb{D})=\{w \in \mathbb{C}:|\arcsin (w -1)|<1\}=:\Omega_s$. Let $\phi:\mathbb{C}^3\times \mathbb{D}\rightarrow \mathbb{C}$ be defined as $\phi(r,s,t;z)=1+\beta_1 s+\beta_2 t$. For $\phi\in \Psi[\Omega_s, 1+\sinh^{-1}z]$, we must have $\phi(r,s,t;z)\notin \Omega_s$. From \cite[Lemma 3.3]{chosine}, we note that the smallest disk containing $\Omega_s$ is $\{\delta \in \mathbb{C}: |\delta -1|<\sinh 1\}$. So,
\begin{equation*}
        |\phi(r,s,t;z)-1|=|\beta_1 s+\beta_2 t|.
\end{equation*}
Proceeding on the same lines of the proof of Theorem \ref{10 thmsigmoid}, we have
\begin{align*}
       |\beta_1 s+\beta_2 t|&\geq n_4(\theta)(\beta_1+\beta_2 n_5(\theta))\\
       &\geq \frac{1}{\sqrt{2}}\bigg(\beta_1-\frac{1}{2}\beta_2\bigg)\\
       &\geq \sinh 1.
\end{align*}
Clearly, $\phi(r,s,t;z)$ lies outside the disk $\{\delta \in \mathbb{C}: |\delta -1|<\sinh 1\}$, which is enough to conclude that $\phi(r,s,t;z)\notin \Omega_s$. Therefore, $\phi\in \Psi[\Omega_s,1+\sinh^{-1} z]$ and the result follows as an application of Lemma \ref{10 millertheorem}.
\end{proof}

\begin{theorem}
Let $\beta_1$, $\beta_2>0$ and $2\beta_1-\beta_2\geq 2\sqrt{2}e$. Let $p$ be analytic function in $\mathbb{D}$ with $p(0)=1$ and
\begin{equation*}
        1+\beta_1 zp'(z)+\beta_2 z^2p''(z)\prec 1+ze^z.
\end{equation*}
Then $p(z)\prec 1+\sinh^{-1}z$.
\end{theorem}

\begin{proof}
Let $h(z)=1+ze^z$ for $z\in \mathbb{D}$, then $\Omega:=h(\mathbb{D})$. Let $\phi:\mathbb{C}^3\times \mathbb{D}\rightarrow \mathbb{C}$ be defined as $\phi(r,s,t;z)=1+\beta_1 s+\beta_2 t$. For $\phi\in \Psi[\Omega, 1+\sinh^{-1} z]$, we need to have $\phi(r,s,t;z)\notin \Omega$. Through \cite[Lemma 3.3]{kumar-ganganiaCardioid-2021},
we observe that the smallest disk containing $\Omega$ is $\{\delta \in \mathbb{C}: |\delta -1|<e\}$. Thus
\begin{equation*}
        |\phi(r,s,t;z)-1|=|\beta_1 s+\beta_2 t|.
\end{equation*}
Analogous to Theorem \ref{10 thmsigmoid}, we have
\begin{align*}
       |\beta_1 s+\beta_2 t|&\geq n_4(\theta)(\beta_1+\beta_2 n_5(\theta))\\
       &\geq \frac{1}{\sqrt{2}}\bigg(\beta_1-\frac{1}{2}\beta_2\bigg)\\
       &\geq e.
\end{align*}
Clearly, $\phi(r,s,t;z)$ lies outside the disk $\{\delta \in \mathbb{C}: |\delta -1|<e\}$ which suffices us to conclude that $\phi(r,s,t;z)\notin \Omega$. Therefore, $\phi\in \Psi[\Omega,1+\sinh^{-1} z]$ and the result follows as an application of Lemma \ref{10 millertheorem}.
\end{proof}

\begin{theorem}
Suppose $\beta_1$, $\beta_2>0$ and $2\beta_1-\beta_2\geq 2\sqrt{2}(e-1)$. Let $p$ be analytic function in $\mathbb{D}$ with $p(0)=1$ and
\begin{equation*}
        1+\beta_1 zp'(z)+\beta_2 z^2p''(z)\prec e^z.
\end{equation*}
Then $p(z)\prec 1+\sinh^{-1}z$.
\end{theorem}

\begin{proof}
Take $h(z)=e^z$ for $z\in \mathbb{D}$. Then, $h(\mathbb{D})=\{w \in\mathbb{C}:|\log w |<1\}=:\Omega$. Let $\phi:\mathbb{C}^3\times \mathbb{D}\rightarrow \mathbb{C}$ be defined as $\phi(r,s,t;z)=1+\beta_1 s+\beta_2 t$. For $\phi\in \Psi[\Omega, 1+\sinh^{-1}z]$, we must have $\phi(r,s,t;z)\notin \Omega$. So,
\begin{equation*}
        |\phi(r,s,t;z)-1|=|\beta_1 s+\beta_2 t|.
\end{equation*}
Similar to the proof of Theorem \ref{10 thmsigmoid}, we have
\begin{align}\label{10 82}
       |\beta_1 s+\beta t|&\geq n_4(\theta)(\beta_1+\beta_2 n_5(\theta))\nonumber\\
       &\geq \frac{1}{\sqrt{2}}\bigg(\beta_1-\frac{1}{2}\beta_2\bigg)\nonumber\\
       &\geq e-1.
\end{align}
Further, we have
\begin{equation*}
     |\log (\phi(r,s,t;z)|=|\log (1+\beta_1 s+\beta_2 t)|.
\end{equation*}
Through Lemma \ref{prilemma4} and \eqref{10 82}, we have
\begin{equation*}
   |\log (1+\beta_1 s+\beta_2 t)|\geq 1, 
\end{equation*}
which implies that $\phi(r,s,t;z)\notin \Omega$. Therefore, $\phi\in \Psi[\Omega,1+\sinh^{-1} z]$ and thus the result follows as an application of Lemma \ref{10 millertheorem}.
\end{proof}

\begin{theorem}
Suppose $\beta_1$, $\beta_2>0$ and $2\beta_1-\beta_2\geq \sqrt{2}\pi$. Let $p$ be analytic function in $\mathbb{D}$ with $p(0)=1$ and
\begin{equation*}
        1+\beta_1 zp'(z)+\beta_2 z^2p''(z)\prec 1+\sinh^{-1} z.
\end{equation*}
Then $p(z)\prec 1+\sinh^{-1}z$.
\end{theorem}

\begin{proof}
Suppose $h(z)=1+\sinh^{-1} z$ for $z\in \mathbb{D}$. Then, $h(\mathbb{D})=\{w \in\mathbb{C}:|\sinh (w -1)|<1\}=:\Omega_{\rho}$. Let $\phi:\mathbb{C}^3\times \mathbb{D}\rightarrow \mathbb{C}$ be defined as $\phi(r,s,t;z)=1+\beta_1 s+\beta_2 t$. For $\phi\in \Psi[\Omega_\rho, 1+\sinh^{-1} z]$, we must have $\phi(r,s,t;z)\notin \Omega_\rho$. Through \cite[Remark 2.7]{kush}, we note that the disk $\{\delta \in\mathbb{C}: |\delta-1|<\pi/2\}$ is the smallest disk containing $\Omega_\rho$. Now
\begin{equation*}
        |\phi(r,s,t;z)-1|=|\beta_1 s+\beta_2 t|.
\end{equation*}
Proceeding on the same lines as in proof of Theorem \ref{10 thmsigmoid}, we have
\begin{align*}
       |\beta_1 s+\beta_2 t|&\geq n_4(\theta)(\beta_1+\beta_2 n_5(\theta))\\
       &\geq \frac{1}{\sqrt{2}}\bigg(\beta_1-\frac{1}{2}\beta_2\bigg)\\
       &\geq \frac{\pi}{2}.
\end{align*}
Clearly, $\phi(r,s,t;z)$ does not belong to the disk $\{\delta \in \mathbb{C}: |\delta -1|<\pi/2\}$, which suffices us to conclude that $\phi(r,s,t;z)\notin \Omega_\rho$. Therefore, $\phi\in \Psi[\Omega_\rho,1+\sinh^{-1} z]$ and thus the result follows as an application of Lemma \ref{10 millertheorem}.
\end{proof}

We conclude this article through the following corollary, obtained from Theorems 3.5-3.10. This is accomplished by considering $p(z) = zf'(z)/f(z)$, as demonstrated in Corollary \ref{10s corollaryfirst}:
\begin{corollary}
Suppose $\beta_1$, $\beta_2>0$ and $f\in\mathcal{A}$. Then, $f\in \mathcal{S}^{*}_{\rho}$ if any of the conditions hold:
\begin{enumerate}[$(i)$]
\item $S_{f}(z)\prec \sqrt{1+z}$ and $(2{\beta_1}- \beta)(2\beta_1-\beta-4\sqrt{2})\geq 8$.
\item $S_{f}(z)\prec (1+Cz)/(1+Dz)$ and $(2\beta_1 -\beta_2)(1-D^2)\geq 2\sqrt{2}(C-D)(1+|D|)$, where $-1<D<C\leq 1$.
\item $S_{f}(z)\prec 2/(1+e^{-z})$ and $2\beta_1-\beta_2\geq 2\sqrt{2} r_0$, where $r_0\approx 0.546302$ is the positive root of the equation $r^2+2 \cot (1)r-1=0$.
\item $S_{f}(z)\prec z+\sqrt{1+z^2}$ and $2\beta_1-\beta_2\geq 4$.
\item $S_{f}(z)\prec 1+\sin z$ and $2\beta_1-\beta_2\geq 2\sqrt{2}\sinh 1$.
\item $S_{f}(z)\prec 1+ze^z$ and $2\beta_1-\beta_2\geq 2\sqrt{2}e$.
\item $S_{f}(z)\prec e^z$ and $2\beta_1-\beta_2\geq 2\sqrt{2}(e-1)$.
\item $S_{f}(z)\prec 1+\sinh^{-1}z$ and $2\beta_1-\beta_2\geq \sqrt{2}\pi$.
\end{enumerate}
\end{corollary}


				%
\subsection*{Acknowledgment}
Neha Verma is thankful to the Department of Applied Mathematics, Delhi Technological University, New Delhi-110042 for providing Research Fellowship.

\end{document}